\documentclass[11pt]{amsart}
\usepackage[utf8]{inputenc}

\usepackage{amsmath, amssymb}
\usepackage{amsthm}
\usepackage[pagebackref, colorlinks, linkcolor=red, citecolor=blue, urlcolor=blue, hypertexnames=true]{hyperref}
\usepackage{amsrefs}
\usepackage{setspace,nicefrac, yhmath, amscd,eucal}
\usepackage{color, ulem}

\newtheorem{thm}{Theorem}[section]
\newtheorem{cor}[thm]{Corollary}
\newtheorem{lem}[thm]{Lemma}
\newtheorem{question}[thm]{Question}
\newtheorem{definition}[thm]{Definition}

\newtheorem{remark}[thm]{Remark}
\newtheorem{proposition}[thm]{Proposition}

\theoremstyle{definition}

\begin{document}

\title[Maximal vN subalgebras from maximal subgroups]{Maximal von Neumann subalgebras arising from maximal subgroups}
\date{}
\author{Yongle Jiang}
\address{Yongle Jiang, School of Mathematical Sciences, Dalian University of Technology, Dalian, 116024, China}
\email{yonglejiang@dlut.edu.cn, jiangyongle@amss.ac.cn}



\maketitle
\begin{abstract}
Ge asked the question whether $LF_{\infty}$ can be embedded into $LF_2$ as a maximal subfactor. We answer it affirmatively by three different approaches, all containing the same key ingredient: the existence of maximal subgroups with infinite index. We also show that point stabilizer subgroups for every faithful, 4-transitive action on an infinite set give rise to maximal von Neumann subalgebras. By combining this with known results on constructing faithful, highly transitive actions, we get many maximal von Neumann subalgebras arising from maximal subgroups with infinite index. 
\end{abstract}

\section{Introduction} 
Let $G$ be any countably infinite discrete group. Denote by $LG$ the group von Neumann algebra of $G$. The starting point for this paper is Question 2 in Ge's problem list \cite{ge} published in 2003 (Note that a maximal subfactor in this paper should be understood as a subfactor which is proper and maximal among all proper von Neumann subalgebras): 
\begin{question}[Ge]\label{question: ge's}
A subfactor (or subalgebra) is called \textbf{maximal} if it is not contained in any proper subalgebra other than itself. Can a non-hyperfinite factor of type II$_1$ have a hyperfinite subfactor as
its maximal subfactor? Can a maximal subfactor of the hyperfinite factor of type II$_1$ have an
infinite Jones index? 
Can $LF_{\infty}$ be embedded into $LF_2$ as a maximal subfactor?
\end{question}

In general, it is not clear whether maximal subalgebras exist for a given von Neumann algebra. But if we restrict our attention to the family of all subalgebras satisfying a certain good property, a maximal element (w.r.t. the partial order defined by inclusion) for this family may exist by Kuratowski-Zorn's lemma. Two such properties are commutativity and amenability.
In fact, maximal abelian von Neumann subalgebras have been studied extensively starting from the very beginning. We refer the readers to the book \cite{masa_book} for the history and classical results on this topic. The first concrete example of abelian maximal amenable von Neumann subalgebra is due to Popa \cite{popa_max injective}. He proved that the abelian subalgebra of the free group factor generated by one of the generators in the free group is maximal amenable, answering a famous question of Kadison \cite{ge}. Later on, people found more concrete examples of maximal amenable von Neumann subalgebras by (modifying) Popa's method, see e.g. \cite{ge_max, shen, cfrw, houdayer, bc_ann}. Quite recently, Boutonnet and Carderi \cite{bc} found an entirely new method to study the question of whether maximal amenable subgroups give rise to maximal amenable von Neumann subalgebras.

By contrast, more than a decade has passed after Question \ref{question: ge's} was posed, it seems neither it nor the topic of maximal subfactors with infinite index has received much attention. On the other hand, recently, Suzuki studied the ``dual" notion of maximal C$^*$-subalgebras, i.e. minimal ambient C$^*$-algebras and obtained many striking results \cite{yuhei, yuhei_19}.


Partially motivated by this situation, together with Skalski, we started exploring the notion of maximal P von Neumann subalgebras in \cite{js}, where P stands for some property even weaker than amenability, e.g. Haagerup property or non-(T). As a by-product of our investigation, we were able to answer the first two parts of Question \ref{question: ge's} affirmatively.  
Inspired by this solution and crucial ingredients used there, we can also settle the last part affirmatively. 

\begin{thm}\label{thm: solve ge's Q}
$LF_{\infty}$ can be embedded as a maximal subfactor into $LF_2$ with infinite Jones index.
\end{thm}


In fact, we will present three proofs and all of them benefit from the same viewpoint: the existence of maximal subgroups with infinite index.

Recall that for any group $G$, a (proper) subgroup $H$ is called a \textbf{maximal subgroup} if any subgroup of $G$ containing $H$ is either $H$ or $G$. The study of  maximal subgroups with infinite index in countably infinite groups was initiated by Margulis and So\u\i fer \cite{ms_1, ms_2, ms_3} in 70's. They proved such subgroups exist in $SL(n, \mathbb{Z})$ for all $n>2$, answering a question of Platonov affirmatively. Central to their proof is the study of dynamical properties for the boundary action of $SL(n, \mathbb{Z})$. Since then, using this boundary technique, people have shown that maximal subgroups of infinite index exist in many other groups, see \cite{ggs} for a recent survey on this topic.

Generally speaking, it is highly nontrivial to study specific properties of a maximal subgroup with infinite index in the above works. This is simply because most known proofs for the existence of such subgroups (at least in linear groups) are non-constructive and rely on Kuratowski-Zorn's lemma.
Nevertheless, it is a fairly well-known fact in group theory that every maximal subgroup with infinite index in a finite rank non-abelian free group is isomorphic to $F_{\infty}$ (see Lemma \ref{lem: maximal subgroups with infinite index in free groups are not finitely generated}). Therefore, to answer Ge's question affirmatively, it suffices to
find a maximal subgroup $H$ with infinite index in $F_2$ such that $LH$ is maximal inside $LF_2$.
This is our initial idea to attack this question.

Next, we briefly discuss the three approaches and other results we get.

In the first approach, to find good maximal subgroups with infinite index in $F_2$, we observe that  for a surjective group homomorphism, the preimage (under this homomorphism) of a maximal subgroup with infinite index in the quotient group is again a maximal subgroup with infinite index in the ambient group. As $F_2$ has abundant quotient groups (i.e. any group with no more than two generators), we have much flexibility to choose good quotient groups with nice maximal subgroups. At the von Neumann algebras level, we decompose the ambient free group factor as a twisted crossed product using the Connes-Jones cocycle and apply a Galois correspondence theorem for such a twisted crossed product. 

In the second approach, we prove directly a Galois correspondence theorem for certain inclusions $LH<LG$ by extracting the key ingredients used in the first approach. This immediately leads to the following more general theorem by applying the well-known work of Olshanskii \cite{ols}.

\begin{thm}\label{thm: i.c.c. hyperbolic groups}
Let $G$ be any torsion free, non-elementary hyperbolic group (e.g. the free group $F_2$). Then there exists a maximal subgroup $H$ with infinite index such that $LH$ is also a maximal von Neumann subalgebra inside $LG$.
\end{thm}

In the last approach, which targets the free group factor case only, we apply Dykema's free decomposition theorem for free group factors \cite{dyk1}, combined with the known Galois correspondence theorem for outer actions by Choda \cite{choda.h} directly.

It is notable that maximal subgroups with infinite index constructed in the first two approaches always contain infinite normal subgroups of the ambient groups, which never happens for many linear groups like $SL_n(\mathbb{Z})$ by Margulis' normal subgroup theorem. Moreover, for a faithful,  2-transitive action (see Definition \ref{def: n-transitive actions} for definitions), the stabilizer subgroup of any point is always maximal but never contains any nontrivial normal subgroups of the ambient group. It is therefore intriguing to ask whether we can prove a similar result for this type of maximal subgroups. This sounds even more plausible since Boutonnet and Carderi have successfully used group dynamics to establish  maximal amenablity for certain von Neumann subalgebras arising from maximal amenable subgroups \cite{bc}. 
Following their dynamical approach in spirit, we found the following simple criterion. 

\begin{thm}\label{thm: faithful 4-transitive actions}
Let $G\curvearrowright X$ be a faithful, 4-transitive action on an infinite set $X$. Let $H$ be the stabilizer subgroup of any point $x$ in $X$. Then $LH$ is a maximal von Neumann subalgebra inside $LG$.
\end{thm}  

A prototype for which this theorem applies is the natural permutation action  $S_{\infty}\curvearrowright \mathbb{N}$, where $S_{\infty}$ denotes the group of permutations of $\mathbb{N}$ with finite support. More generally, we can apply known results on constructing faithful, highly transitive (hence also 4-transitive) actions for various groups to get the following statement directly.

\begin{cor}\label{cor: acylindricially hyperbolic and t.full groups}
Let $G$ be either a countably acylindrically hyperbolic group with trivial finite radical or the topological full group of a minimal \'etale groupoid over the Cantor set. Then there exists some maximal subgroup $H$ with infinite index such that $LH$ is also a maximal von Neumann subalgebra inside $LG$. Moreover, $H$ can be chosen to contain no nontrivial normal subgroups of $G$.
\end{cor}

Note that this corollary provides unified new solutions to the last two parts of Question \ref{question: ge's} by taking $G=S_{\infty}$ and $F_2$ respectively. Furthermore, by modifying the proof of Theorem \ref{thm: faithful 4-transitive actions}, we will prove this theorem still holds for a particular faithful, 3-transitive action (see Proposition \ref{prop: maximal for PSL_2(Q)}). This also gives a new solution to the first part of Question \ref{question: ge's}. However, we do not know whether this part still has a positive answer if we further assume the ambient non-hyperfinite II$_1$ factor is non-thin (e.g. the free group factor $LF_n$ for $n\geq 3$) in the sense of Ge-Popa \cite{gepopa}. See Subsection \ref{subsection: two endding questions} for more discussion on this.

We refer the readers to \cite{dgo,osin_book} for background of relative/acylindrically hyperbolic groups and other related notions and to \cite{kr_1, kr_2} for basic facts on von Neumann algebras.

\paragraph{\textbf{On the method}}Let $H$ be a maximal subgroup in $G$. It is not hard to observe that $LH$ is maximal in $LG$ iff for any von Neumann subalgebra $P$ containing $LH$, we have $u_g^*E(u_g)\in \mathbb{C}1$ for all $g\in G$, where $E: LG\twoheadrightarrow P$ denotes the conditional expectation. Our method is based on this observation and we just try to find sufficient conditions on $H$ and $G$ to guarantee $u_g^*E(u_g)\in \mathbb{C}1$ for all $g\in G$.

The rest of the paper is organized as follows. In Section \ref{section: preliminaries}, we prepare some facts on twisted crossed product in the first approach. 
Section \ref{section: three approaches} is splitted into 3 subsections, each consisting of one proof of Theorem \ref{thm: solve ge's Q}. In Subsection \ref{subsection: 2nd approach}, we also prove a general version of Theorem \ref{thm: i.c.c. hyperbolic groups}, i.e. Theorem \ref{thm: general version of i.c.c. hyperbolic groups}. In Section \ref{section: transitive actions}, we first present an example showing maximal subgroups do not necessarily generate maximal von Neumann subalgebras. Then we prove Theorem \ref{thm: faithful 4-transitive actions} and Corollary \ref{cor: acylindricially hyperbolic and t.full groups}: establishing maximality for von Neumann subalgebras arising from a distinct class of maximal subgroups related to faithful, 4-transitive actions. 
In Section \ref{section: remarks}, we first observe maximal von Neumann subalgebras which appeared in both Theorem \ref{thm: general version of i.c.c. hyperbolic groups} and Theorem \ref{thm: faithful 4-transitive actions} are also rigid in the sense of Longo \cite{longo}. Then we present two other (counter-)examples of maximal von Neumann subalgebras which do not seem to fit into our main theorems. In particular, we prove Proposition \ref{prop: maximal for PSL_2(Q)}, which is an example of faithful, 3-transitive action such that the stabilizer subgroups still generate maximal von Neumann subalgebras. Then we finish this paper by asking two questions on free groups and free group factors.

\paragraph{\textbf{Notations}} We fix some notations used later.
\begin{itemize}
\item For a von Neumann algebra $N$, $Aut(N)$, $\mathcal{Z}(N)$, $\mathcal{U}(N)$, $\mathcal{P}(N)$, $\mathfrak{F}(N)$, $N^t$ denotes its automorphism group, center, unitary group, the set of projections, fundamental group, $t$-amplification respectively. 
\item We usually denote groups by letters like $G$, $H$, $K_0$, $K$ and $\Gamma$.
\item For groups $H<G$, $H\backslash G/H$ denotes the double coset space.
\item $E_P$ or $E$ usually denotes the conditional expectation from the ambient group von Neumann algebra onto an intermediate von Neumann subalgebra $P$, which is normal and trace preserving.
\end{itemize}

\section{Preliminaries}\label{section: preliminaries}
\subsection{Twisted crossed product}
Let us recall the standard construction of twisted crossed product.
For more on cocycle actions/twisted crossed product, see \cite{choda.m, sut, cj,  popa_06, cam}\cite[p. 9]{ccssww}.

First, let us recall the definition of cocycle actions.

Let $N$ be a von Neumann algebra acting on a Hilbert space $\mathcal{H}$. A \textbf{cocycle action} of a discrete group $K$ on $N$ is a pair $(\sigma, v)$, where $\sigma: K\to Aut(N)$ and $v: K\times K\to \mathcal{U}(N)$ satisfy the following conditions for all $k, l, m\in K$,
\begin{eqnarray*}
\sigma_k\sigma_l&=&ad(v(k, l))\sigma_{kl},\\
v(k, l)v(kl,m)&=&\sigma_k(v(l, m))v(k, lm),\\
v(1, l)&=&v(l, 1)=1.
\end{eqnarray*}

The \textbf{twisted crossed product} of $N$ by $K$, say $N\rtimes_{(\sigma, v)}K$, is then defined as the von Neumann algebra 
acting on $\ell^2(K,\mathcal{H})$ generated by $\pi_{\sigma}(N)$ and $\lambda_v(K)$, where $\pi_{\sigma}$ is the faithful 
normal representation of $N$ on $\ell^2(K, \mathcal{H})$ defined by 
$$(\pi_{\sigma}(x)\xi)(l)=\sigma_{l^{-1}}(x)\xi(l),$$
while, for each $k\in K$, $\lambda_v(k)$ is the unitary operator on $\ell^2(K,\mathcal{H})$ defined by 
$$(\lambda_v(k)(\xi))(l)=v(l^{-1}, k)\xi(k^{-1}l), ~ \forall x\in N, \xi\in\ell^2(K,\mathcal{H}), l\in K.$$

One easily checks that the covariance formula 
$$\pi_{\sigma}(\sigma_k(x))=ad(\lambda_v(k))(\pi_{\sigma}(x))$$
holds for all $k\in K, x\in N$, and also that 
$$\lambda_v(k)\lambda_v(l)=\pi_{\sigma}(v(k, l))\lambda_v(kl)$$
holds for all $k, l\in K$.

Next, assume $N$ is a II$_1$ factor, recall that a cocycle action is \textbf{free} if $\sigma_k$ is \textbf{properly outer}, i.e. $\sigma_k$ cannot be implemented by unitary elements in $N$ for all $k\neq e$ in $K$. We can check that the cocycle action is free iff $N'\cap (N\rtimes_{(\sigma, v)}K)=\mathbb{C}$, i.e. $N$ is \textbf{irreducible} in $N\rtimes_{(\sigma, v)}K$.

We record a well-known proposition below, c.f. \cite[Theorem 11]{choda.m}\cite[II, Proposition 3.17]{sut} \cite[Proposition 3]{bedos}.

\begin{proposition}\label{prop: realize ambient group factor as a twisted crossed product}
Let $1\to H\to G\overset{\pi}{\to} K\to 1$ be an exact sequence of discrete groups. Then there exists a cocycle action $(\sigma, v)$ of $K$ on $LH$ such that $LG\cong LH\rtimes_{(\sigma, v)}K.$
\end{proposition}
\begin{proof}[Sketch of proof]
First, let us construct $(\sigma, v)$ in order to form the twisted crossed product.
For each $k\neq 1$ in $K$, choose a $n_k\in G$ such that $\pi(n_k)=k$ and set $n_e=1$. Define $\sigma: K\to Aut(H)$ by 
$\sigma_k(h)=n_khn_k^{-1}(h\in H)$, and $v: K\times K\to H$ by $v(k, l)=n_kn_ln_{kl}^{-1} (k, l\in K)$. Then use $(\sigma, v)$ to form the twisted crossed product.

To show $LG\cong LH\rtimes_{(\sigma, v)}K$, it suffices to construct a unitary operator $W: \ell^2(K, \mathcal{H})\cong \ell^2(G)$ which intertwines the two von Neumann algebras, where $\mathcal{H}=\ell^2(H)$.

For $\xi\in \ell^2(K, \mathcal{H})$, define $(W\xi)(g)=\xi(\pi(g))(n_{\pi(g^{-1})}g)$. It is routine to check this $W$ is what we need.
\end{proof}

\subsection{Connes-Jones cocycles}

Let $\Gamma=\langle S \rangle$ be an infinite group and $\pi: F_S\to \Gamma\to 1$ with kernel $ker(\pi)\cong F_{\infty}$ (this holds if $(\Gamma,S)\neq (F_S, S)$). 

Then, we can check that $N:=L(ker(\pi))<M:=LF_S$ is irreducible and regular and $\Gamma=\mathcal{N}_M(N)/\mathcal{U}(N)$, see \cite{choda.m}. 

This gives rise to $LF_S\cong N\rtimes_{(\sigma, v)} \Gamma$ for some free cocycle action $(\sigma,v)$ of $\Gamma$ on $N$ as mentioned above. This cocycle is the so-called \textbf{Connes-Jones cocycle} \cite{cj} \cite[Appendix A.2]{popa_06}.



\section{Three approaches to prove Theorem \ref{thm: solve ge's Q}}\label{section: three approaches}

In this section, we give three proofs for Theorem \ref{thm: solve ge's Q}.

\subsection{First approach}\label{subsection: first approach}

From now on, we take $\Gamma$ to be a finitely generated infinite non-free group which contains a maximal subgroup $K$ with infinite index. For example, one may take $\Gamma$ to be the Houghton group $H_n$ for any $n\geq 2$, see \cite[Example 3.6]{cornulier} and references therein. 
We will work with two short exact sequences, one is $1\to ker(\pi)\to F_S\overset{\pi}{\to} \Gamma\to 1$ as mentioned in the previous section and the other one is 
$1\to ker(\pi)\to K'\overset{\pi}{\to}K\to 1$, where $K'=\pi^{-1}(K)$.

Now, we prepare some lemmas. First, we record a well-known fact on free groups.

\begin{lem}\label{lem: maximal subgroups with infinite index in free groups are not finitely generated}
Every maximal subgroup with infinite index in a finite rank nonabelian free group is isomorphic to $F_{\infty}$.
\end{lem}
\begin{proof}
Let $H$ be a maximal subgroup in $F_n$ for $1<n<\infty$ with $[F_n: H]=\infty$. By \cite{hall}, $F_n$ satisfies the following property (called Hall's property):

If $K$ is a finitely generated subgroup of $F_n$, then there exists some subgroup $K_0$ of $F_n$ such that $[F_n: K*K_0]<\infty$. 

Assume $H$ is finitely generated, then take $K=H$ in Hall's property. We deduce $K_0\neq \{e\}$ as $[F_n: H]=\infty$. But this contradicts the fact that $H$ is maximal in $F_n$ since $H\lneq H*\langle s^2 \rangle\lneq H*K_0<F_n$ for any nontrivial element $s\in K_0$. Hence $H$ is not finitely generated and therefore $H\cong F_{\infty}.$
\end{proof}

\begin{lem}\label{lem: lift of maximal subgroup is still maximal}
$K'<F_S$ is a maximal subgroup with infinite index and $K'\cong F_{\infty}$.
\end{lem}

\begin{proof}
Assume $K'<L<F_S$ for some subgroup $L$. By applying $\pi$ to these groups, we get $K<\pi(L)<\Gamma$. As $K<\Gamma$ is maximal, we deduce that either 
$K=\pi(L)$ or $\pi(L)=\Gamma$. If $K=\pi(L)$, then $L<\pi^{-1}(K)=K'$, therefore $K'=L$; if $\pi(L)=\Gamma$, then $F_S<ker(\pi)L=L$, i.e. $L=F_S$. Clearly, $[F_S: K']=[\Gamma, K]=\infty$.
To see $K'\cong F_{\infty}$, just apply Lemma \ref{lem: maximal subgroups with infinite index in free groups are not finitely generated}.
\end{proof}

The following fact was proved in \cite[Theorem 2]{nt} for normalized 2-cocycles (i.e. $v_{g, g^{-1}}=id$ for all $g\in \Gamma$), but this assumption is not used in the proof. We include the proof for completeness.

\begin{lem}[Galois correspondence for cocycle actions]\label{lem: twisted Galois correspondence}
Let $\Gamma\overset{(\sigma, v)}{\curvearrowright} N$ be a free cocycle action, where $N$ is a II$_1$ factor. Then every intermediate von Neumann subalgebra $P$ with $N<P<N\rtimes_{(\sigma, v)}\Gamma$ is of the form $P=N\rtimes_{(\sigma, v)}K_0$ for some subgroup $K_0<\Gamma$.
\end{lem}  
\begin{proof}
The proof is similar to the genuine action case given by Choda \cite{choda.h}.  Indeed, for all $a\in\mathcal{U}(N)$, we have that $a$ commutes with $u_s^*E_P(u_s)$ for all $s\in \Gamma$.
Indeed, $au_s^*E_P(u_s)=u_s^*u_sau_s^*E_P(u_s)=u_s^*\sigma_s(a)E_P(u_s)=u_s^*E_P(\sigma_s(a)u_s)=u_s^*E_P(u_sa)=u_s^*E_P(u_s)a$.

Therefore, $u_s^*E_P(u_s)\in N'\cap N\rtimes_{\sigma,v}\Gamma=\mathbb{C}$ since $(\sigma, v)$ is free if and only if $N'\cap N\rtimes_{(\sigma, v)}\Gamma=\mathbb{C}$, i.e. $N$ is irreducible in $N\rtimes_{(\sigma, v)}\Gamma$. Therefore, $E_P(u_s)=\lambda_su_s$ for some $\lambda_s\in\mathbb{C}$ and all $s\in \Gamma$. 

Define $K_0=\{s\in\Gamma: u_s\in P\}$. We check that $K_0$ is a subgroup of $\Gamma$.

Indeed, $e\in K_0$ as $u_e=id\in P$. If $s, t\in K_0$, then $st\in K_0$ since $u_su_t=v(s, t)u_{st}$ implies $u_{st}=v(s, t)^{-1}u_su_t\in P$. Similarly, if $s\in K_0$, then $u_{s^{-1}}u_s=v(s^{-1}, s)u_e$ implies $u_{s^{-1}}=v(s^{-1}, s)u_s^*\in P$, hence $s^{-1}\in K_0$. Therefore, $K_0$ is a group.

Clearly, $\lambda_s\neq 0\Rightarrow s\in K_0$. This implies (using also that $E_P$ is normal) that $P=E_P(N\rtimes_{(\sigma, v)}\Gamma)<N\rtimes_{(\sigma, v)}K_0<P$. Hence, $P=N\rtimes_{(\sigma, v)}K_0$.
\end{proof}

\begin{lem}\label{lem: compatibility for twisted subalgebras}
Let $\phi: LF_S\cong L(ker(\pi))\rtimes_{(\sigma, v)}\Gamma$ be an isomorphism appearing in Proposition \ref{prop: realize ambient group factor as a twisted crossed product}. Then $\phi(LK')=L(ker(\pi))\rtimes_{(\sigma, v)}K$.
\end{lem}

\begin{proof}
Notice we have a short exact sequence $1\to ker(\pi)\to K'\overset{\pi|_{K'}}{\to} K\to 1$. Recall the definition of $(\sigma, v)$ which appeared in the proof of 
Proposition \ref{prop: realize ambient group factor as a twisted crossed product} depends on a fixed section map for $\pi$, so one can use the same section map, but restricted to $K$. This means we use the same $(\sigma, v)$ but for this new exact sequence. Then the same proof shows $LK'\cong L(ker(\pi))\rtimes_{(\sigma, v)}K$.
\end{proof}

We are now ready to prove Theorem \ref{thm: solve ge's Q}. Recall that $1\to ker(\pi)\to F_S\overset{\pi}{\to}\Gamma\to 1$ and $1\to ker(\pi)\to K'\overset{\pi}{\to}K\to 1$ are two short exact sequences. Here, $\Gamma$ is a finitely generated infinite non-free group which contains a maximal subgroup $K$ with infinite index and $K'=\pi^{-1}(K)\cong F_{\infty}$ by Lemma \ref{lem: maximal subgroups with infinite index in free groups are not finitely generated}.

\begin{proof}[First proof of Theorem \ref{thm: solve ge's Q}]
As $K<\Gamma$ is a maximal subgroup and $(\sigma, v)$ is a free cocycle action, we deduce that $L(ker(\pi))\rtimes_{(\sigma,v)} K<L(ker(\pi))\times_{(\sigma, v)} \Gamma$ is maximal by Lemma \ref{lem: twisted Galois correspondence}. Then by Lemma \ref{lem: compatibility for twisted subalgebras}, we get $LF_{\infty}\cong LK'<LF_S$ is maximal.

It is possible to assume $|S|=2$ directly by taking $\Gamma$ to be an infinite simple group with two generators, e.g. a Tarski monster group \cite{ol_1, ol_2}. Nevertheless, we can also argue as follows:

Voiculescu's amplification formula \cite{voi} tells us that
$LF_2\cong LF_{|S|}\bar\otimes M_{\sqrt{|S|-1}}(\mathbb{C})$. 
Note that we can enlarge $|S|$ so that $\sqrt{|S|-1}$ is an integer.
Therefore, we know that 
\begin{align*}
LF_{\infty}\cong L(F_{\infty})\bar\otimes M_{\sqrt{|S|-1}}(\mathbb{C})<LF_S\bar\otimes M_{\sqrt{|S|-1}}(\mathbb{C})\cong LF_2
\end{align*}
is maximal by Ge-Kadison's splitting theorem \cite{gekadison} or a simple calculation. Here, the first isomorphism is based on the fact that the fundamental group of $LF_{\infty}$ is equal to $\mathbb{R}_{+}^{*}$, as proved in \cite{rad_92} by Radulescu, or just $\mathbb{Q}^+\setminus\{0\}\subseteq \mathfrak{F}(LF_{\infty})$ as proved by Voiculescu in \cite{voi}. 
\end{proof}

\subsection{Second approach}\label{subsection: 2nd approach}

\begin{proof}[Second proof of Theorem \ref{thm: solve ge's Q}]
It suffices to observe that we can completely avoid the use of twisted cocycle actions from the first approach by isolating the crucial properties used there or isolating the crucial properties used in the proof of Choda's Galois correspondence theorem for outer actions on II$_1$ factors \cite{choda.h}. More precisely, we only need to observe the following lemma.

\begin{lem}\label{lem: 2nd approach}
Let $H_0<H<G$ be countable discrete groups. Assume the following conditions hold.
\begin{itemize}
\item[(1)] $gH_0g^{-1}<H$ for all $g\in G$;
\item[(2)] $H_0<G$ is \textbf{relative I.C.C.}, i.e. $\sharp\{kgk^{-1}: k\in H_0\}=\infty$ for all $e\neq g\in G$.
\end{itemize}
Let $LH<P<LG$ be an intermediate von Neumann subalgebra. Then $P=LJ$ for some intermediate subgroup $H<J<G$. 
\end{lem}
\begin{proof}
Denote by $E_P: LG\twoheadrightarrow P$ the conditional expectation onto $P$. Then for all $g\in G$, $u_g^*E_P(u_g)\in LH_0'\cap LG$ by condition (1), which equals $\mathbb{C}1$ by condition (2). Therefore, $E_P(u_g)=c_gu_g$ for some $c_g\in\mathbb{C}$ for all $g\in G$. Then define $J=\{g\in G: u_g\in P\}$. One can check that $H<J<G$ is a subgroup and $LJ<P=E_P(LG)<LJ$, hence $P=LJ$.
\end{proof}
Now, following the notation in Subsection \ref{subsection: first approach}, we can apply the above lemma to $G=F_S, H=\pi^{-1}(K)$ and $H_0=ker(\pi)$. As $H<G$ is maximal, we deduce that $LH<LG$ is maximal. Then we can proceed as the first approach to finish the proof by arguing we can assume $|S|=2$. 
\end{proof} 

\begin{remark}
We will use a local version of Lemma \ref{lem: 2nd approach} in the proof of Proposition \ref{prop: sl_2z is maximal in sl_2(z[1/p])}.
Special cases of Lemma \ref{lem: 2nd approach} are known before, see e.g. \cite[Corollary 3.8]{cd}. 
\end{remark}

Next, we observe that the above approach actually works in a more general context: certain acylindrically hyperbolic groups (see \cite{dgo, osin}). More precisely, we have the following general version of Theorem \ref{thm: solve ge's Q}. 

\begin{thm}\label{thm: general version of i.c.c. hyperbolic groups}
Let $G$ be a torsion free finitely presented properly relative hyperbolic group, e.g. a torsion free, non-elementary hyperbolic group. Then there exists a maximal subgroup $H$ of $G$ with infinite index such that $LH$ is maximal inside $LG$.
\end{thm}

\begin{proof}
It suffices to argue the following ingredients are available in our context.
\begin{itemize}
\item[(1)] There exists a quotient of $G$, say $\Gamma$ which contains a maximal subgroup with infinite index.
\item[(2)] $K:=Ker(G\twoheadrightarrow \Gamma)<G$ is relative I.C.C.
\end{itemize}

The first condition is known by \cite[Corollary 1.7]{amo}. In the non-elementary hyperbolic groups setup, we can use directly the deep work of Olshanskii in \cite{ols}. In both contexts, the authors proved that if $G$ is non-elementary (which is the case if $G$ is assumed to be I.C.C., in particular, torsion free by \cite[Theorem 8.14]{dgo}), then $G$ has a non-trivial finitely presented quotient $\Gamma$ without proper subgroups of finite index. Therefore, we can always find a maximal subgroup of $\Gamma$ with infinite index by Kuratowski-Zorn's lemma. Moreover, note that $K$ is infinite as it is nontrivial and $G$ is I.C.C.

The second condition is a standard fact based on the north-south dynamics for loxodromic elements, which always exist in the infinite normal subgroup $K$. More precisely, we have the following lemma, which should be well-known to experts. 
\end{proof}

\begin{lem}
Let $G$ be a torsion free acylindrically hyperbolic group and let $K\lhd G$ be infinite. Then $K<G$ is relative I.C.C.
\end{lem}
\begin{proof}
First, one needs to argue that $K$ contains a loxodromic element (See \cite[Lemma 4.7]{gg} for a proof of a special case). By \cite[Corollary 1.5]{osin}, we know that the class of acylindrically hyperbolic groups is closed under taking s-normal, in particular, infinite normal subgroups. Hence $K$ itself is acylindrically hyperbolic and hence contains loxodromic elements (independent of looking at them from $K$ or $G$). 

Now, the proof is essentially the same as \cite[Theorem 8.14 $(b)\Rightarrow (c)$]{dgo}. Let us record it here. 
Let $c$ be one loxodromic element in $H$. Note that we can assume $c$ is also a WPD element by \cite[AH3 in Theorem 1.2]{osin}. 
Recall that $E(c)$, the unique maximal virtually cyclic subgroup containing $c$ (\cite[Corollary 2.9]{dgo}), is a hyperbolic embedded subgroup of $G$, i.e. 
$E(c)\hookrightarrow_hG$ by \cite[Theorem 6.8]{dgo}. Now, we follow the proof of \cite[Theorem 8.14]{dgo}.   

Let $g\in G\setminus \{1\}$. If $\{c^{-n}gc^n: n\in\mathbb{Z}\}$ is infinite, then we are done. Otherwise, $c^{-m}gc^m=g$ for some $m\in\mathbb{N}$. Hence $|g^{-1}E(c)g\cap E(c)|=\infty$ and $g\in E(c)$ by Proposition 4.33 in \cite{dgo}. 
Now, if $h^{-1}gh=f^{-1}gf$ for some $f, h\in G$, then $(fh^{-1})^{-1}g(fh^{-1})=g$. Hence
$g\in (fh^{-1})^{-1}E(c)(fh^{-1})\cap E(c)$. As $G$ is torsion free, we know $g$ has infinite order. Hence $|(fh^{-1})^{-1}E(c)(fh^{-1})\cap E(c)|=\infty$ and $fh^{-1}\in E(c)$. As a group which contains a non-degenerate hyperbolically embedded subgroup is non-elementary, the index of $E(c)$ in $G$ is infinite. To finish the proof, it suffices to argue that $[K: K\cap E(c)]=\infty$. 

Assume the index is finite, then $[K:\langle c\rangle]=[K: K\cap E(c)][K\cap E(c): \langle c\rangle]\leq [K: K\cap E(c)][E(c): \langle c\rangle]<\infty$, i.e. $K$ is virtually cyclic, a contradiction as acylindrically hyperbolic groups always contain non-abelian free groups.
\end{proof}

\subsection{Third approach}\label{subsection: 3rd proof}

\begin{proof}[Third proof of Theorem \ref{thm: solve ge's Q}]
Recall that Dykema proved in \cite{dyk1} that $LF_n\cong *_{i=1}^nLH_i$, where $H_i$ are any infinite amenable groups and $2\leq n\leq \infty$. In particular, we have $LF_2\cong L\mathbb{Z}*LH\cong L(\mathbb{Z}*H)=L((*_H\mathbb{Z})\rtimes H)=(*_HL\mathbb{Z})\rtimes H$ and $N:=L((*_H\mathbb{Z})\rtimes K)$ is a maximal von Neumann subalgebra if $K<H$ is a maximal subgroup by Choda's Galois correspondence theorem \cite{choda.h} for the outer action $H\curvearrowright *_HL\mathbb{Z}$. By taking suitable $H$ and $K$ (i.e. $[H:K]=\infty$, say, $H=S_{\infty}, K=Stab_H(\{1\})\cong S_{\infty\setminus \{1\}}\cong S_{\infty}$), one can make sure $[LF_2: N]=\infty$. Now, it suffices to check that $N\cong LF_{\infty}$. Indeed, write $H=\sqcup_{i\in\mathbb{Z}}Kh_i$, then $L((*_H\mathbb{Z})\rtimes K)=L((*_{i\in\mathbb{Z}}(*_{Kh_i}\mathbb{Z}))\rtimes K)\cong L((*_{i\in\mathbb{Z}}(*_K\mathbb{Z}))\rtimes K)\cong L((*_{K}(*_{i\in\mathbb{Z}}\mathbb{Z}))\rtimes K)\cong L((*_{K}F_{\infty})\rtimes K)\cong L(F_{\infty}*K)\cong LF_{\infty}*LK\cong LF_{\infty}$ by Dykema's above result.
\end{proof}

To round out this section, let us mention that one can also prove $LF_{\infty}$ can be embedded as a maximal subfactor inside the interpolated free group factor $LF_t$ for all $1<t<\infty$ \cite{dyk2, rad_94}. Indeed, it is an easy corollary of the amplification formula $(LF_2)^{1/\sqrt{t-1}}\cong LF_{t}$, $\mathfrak{F}(LF_{\infty})=\mathbb{R}_+^*$ and Theorem \ref{thm: solve ge's Q} once we establish the following proposition.

\begin{proposition}
Let $N<M$ be II$_1$ factors. Then the following are equivalent.\\
(1) $N<M$ is maximal.\\
(2) $pNp<pMp$ is maximal for all nonzero projection $p\in N$.\\
(3) $pNp<pMp$ is maximal for some nonzero projection $p\in N$.\\
(4) $N^t<M^t$ is maximal for all $1<t<\infty$.\\
(5) $N^t<M^t$ is maximal for some $1<t<\infty$.
\end{proposition}
\begin{proof}
We prove $(1)\Longleftrightarrow(2)\Longleftrightarrow(3)$ below, $(1)\Longleftrightarrow(4)\Longleftrightarrow(5)$ will follow from this proof.

$(1)\Rightarrow (2)$: If $\tau(p)=1/n$ for some positive integer $n$, then $M\cong M_n(\mathbb{C})\bar{\otimes}pMp$ (and similarly for $N$) implies that $pNp<pMp$ is maximal. For the other cases, assume $pNp<pMp$ is not maximal, hence $pNp\lneq A\lneq pMp$ for some von Neumann algebra $A$. Then, take any nonzero projection $q\in pNp$ such that $\tau_{pMp}(q)=1/(n\tau(p))<1$ for some large enough positive integer $n$. Note that  
$qNq<qMq$ is maximal as $\tau(q)=1/n$. Hence, $qNq=qAq$ or $qAq=qMq$. In both cases, we can take a nonzero projection $q'\in qNq$ such that $\tau_{pMp}(q')=1/m$ for some $m>1$. Then, we still have $q'Nq'=q'Aq'$ or $q'Aq'=q'Mq'$. But as $pMp\cong q'Mq'\bar{\otimes}M_m(\mathbb{C})$ (similarly for $pAp$ and $pNp$), we deduce that $pNp=pAp$ or $pAp=pMp$, a contradiction. 

$(2)\Rightarrow (3)$: this is trivial.

$(3)\Rightarrow (1)$: By the proof of $(1)\Rightarrow (2)$, we may assume $\tau(p)=1/n$ for some positive integer $n$ after replacing $p$ by a smaller projection. Then $M\cong M_n(\mathbb{C})\bar{\otimes}pMp$ (and similarly for $N$) implies $N<M$ is maximal by a simple calculation or Ge-Kadison's splitting theorem \cite{gekadison}.
\end{proof}



\section{Maximal subalgebras from faithful, 4-transitive actions}\label{section: transitive actions}

In this section, we prove Theorem \ref{thm: faithful 4-transitive actions} and Corollary \ref{cor: acylindricially hyperbolic and t.full groups}. Let us first explain the motivation behind Theorem \ref{thm: faithful 4-transitive actions}.

In general, if $H<G$ is a maximal subgroup, $LH<LG$ may not be maximal. Indeed, it is easy to check that $G$ is a simple group iff the diagonal subgroup $\Delta(G):=\{(g, g): g\in G\}$ is maximal inside $G\times G$. Nevertheless, $L(\Delta(G))$ is not maximal in $L(G\times G)$ for any nontrivial $G$. 

To see this, let $x=u_{(s, e)}+u_{(e, s)}\in L(G\times G)$, where $s$ is any nontrivial element in $G$. Denote by $\phi$ the automorphism of $L(G\times G)$ induced from the flip automorphism on $G\times G$, i.e. $\phi(s, t)=(t, s)$ for all $(s,t)\in G\times G$.   
Clearly, $L(\Delta(G))\lneq Fix(\phi)\lneq L(G\times G)$, where $Fix(\phi):=\{x\in L(G\times G): \phi(x)=x\}$. Indeed, $x\in Fix(\phi)\setminus L(\Delta(G))$ and $u_{(e, s)}\in L(G\times G)\setminus Fix(\phi)$. For a different example, see Proposition \ref{prop: upper trangular matices is not maximal}.
 
Motivated by this example, it is natural to ask the following question.

\begin{question}\label{question: find sufficient conditions}
Let $G$ be a countable discrete group and $H$ be a maximal subgroup with infinite index. Find sufficient conditions on $G$ and $H$ such that $LH$ is maximal inside $LG$.
\end{question}

Clearly, it is equivalent to asking for conditions on $H<G$ such that $u_g^*E_P(u_g)\in\mathbb{C}1$ for all $g\in G$, where $P$ is any von Neumann subalgebra containing $LH$. Lemma \ref{lem: 2nd approach} shows one sufficient condition is to assume there exists some von Neumann subalgebra $B<LH$ such that $u_gBu_g^*<LH$ for all $g\in G$ and $B'\cap LG=\mathbb{C}$. And a typical choice for this $B$ is $LK$ for some nontrivial normal, relative I.C.C. subgroup $K$ in $G$, just as we did in the first two approaches.

Nevertheless, maximal subgroups do not necessarily contain nontrivial normal subgroups of the ambient groups. Indeed, $\Delta(G)<G\times G$ for any simple nontrivial group $G$ is already such an example. In fact, \cite[Remark 8.8]{gg} shows that those maximal subgroups inside $SL_n(\mathbb{Z})$ constructed using boundary techniques are also of this type. Therefore, it is desirable to study Question \ref{question: find sufficient conditions} for maximal subgroups containing no non-trivial normal subgroups of the ambient groups.

Motivated by this question, we found such a necessary condition as in Theorem \ref{thm: faithful 4-transitive actions}. To prove it, let us first recall the notion of $n$-transitive actions. 
\begin{definition}\label{def: n-transitive actions}
Let $n$ be any positive integer. An \textbf{action} of $G$ on a set $X$ is a group homomorphism from $G$ into the symmetry group of $X$. It is \textbf{faithful} if this homomorphism is injective, i.e. one can identify $G$ as a subgroup of $Sym(X)$. It is called \textbf{$n$-transitive} (resp. \textbf{$n$-sharply transitive}) if $|X|\geq n$ and for any two $n$-tuples of distinct points in $X$, say 
$\bar{\alpha}=(\alpha_1,\ldots,\alpha_n)$ and $\bar{\beta}=(\beta_1,\ldots,\beta_n)$, there exists an (resp. a unique) element $g\in G$ such that $g\bar{\alpha}=\bar{\beta}$, i.e. $g\alpha_i=\beta_i$ for all $1\leq i\leq n$. It is \textbf{highly transitive} if it is $n$-transitive for all $n\geq 1$.
\end{definition}
Let us record some easy facts concerning the above definition.
\begin{itemize}
\item[(1)] If $|X|=\infty$ and $n\geq 1$, then any $(n+1)$-transitive action is not $n$-sharply transitive.
\item[(2)] It is easy to check that an infinite group $G$ is I.C.C. if it admits a faithful, 2-transitive action, see e.g. \cite[Lemma 4.2]{ho}. 
\item[(3)] For a faithful, 2-transitive action, the stabilizer subgroup of any point is a maximal subgroup (as the cardinality of the double coset space is two) and does not contain any nontrivial normal subgroups of the ambient group $G$.
\end{itemize}

Now, we are ready to prove Theorem \ref{thm: faithful 4-transitive actions}.
\begin{proof}[Proof of Theorem \ref{thm: faithful 4-transitive actions}]
Let $P$ be any intermediate von Neumann subalgebra between $LH$ and $LG$. Let $E: LG\twoheadrightarrow P$ be the conditional expectation. Fix any $g\in G\setminus H$, then clearly
$u_g^*E(u_g)\in L(g^{-1}Hg\cap H)'\cap LG$.

Note that $g^{-1}Hg\cap H=Stab(x)\cap Stab(g^{-1}x)$. As $g\not\in H$, we know $g^{-1}x\neq x$. 
Since the action is 2-transitive, $G=H\sqcup HgH$ holds. 

Denote by $\sigma$ the involution on $X$ such that $\sigma(x)=g^{-1}x$, $\sigma(g^{-1}x)=x$ and 
$\sigma(z)=z$ for all $z\in X\setminus \{x, g^{-1}x\}$.

\textbf{Claim}. If $\sigma\in G$, then $L(g^{-1}Hg\cap H)'\cap LG\subseteq L(\{id, \sigma\})$; otherwise, $L(g^{-1}Hg\cap H)'\cap LG=\mathbb{C}$.

\begin{proof}[Proof of the Claim]
We need to show that if $G\ni \tau\not\in \{id, \sigma\}$ (or $id\neq \tau\in G$ if $\sigma\not\in G$), then 
\[\sharp\{t\tau t^{-1}: t\in g^{-1}Hg\cap H\}=\infty.\]

If $id\neq \tau\in g^{-1}Hg\cap H$, then there exists some $y\in X\setminus \{x, g^{-1}x\}$ such that $y\neq \tau(y)$  as the action is faithful. Now take any infinite sequence $\{y_i\}\subseteq X\setminus\{x, g^{-1}x, y\}$. There exists some $t_n\in g^{-1}Hg\cap H$ such that $t_n(x, g^{-1}x, y, \tau(y))=(x, g^{-1}x, y, y_n)$ for each $n$ as the action is 4-transitive and $X$ is infinite. Then as $t_n(\tau(y))\neq t_m(\tau(y))$, we deduce $t_m^{-1}t_n\tau(t_m^{-1}t_n)^{-1}y\neq \tau(y)$ and hence $t_n\tau t_n^{-1}\neq t_m\tau t_m^{-1}$ if $n\neq m$.

Now we assume $\tau\not\in g^{-1}Hg\cap H$ and split the proof into three cases depending on $\sharp(\tau\{x, g^{-1}x\}\cap \{x, g^{-1}x\})=0$, 1 or 2, where $\tau\{x, g^{-1}x\}:=\{\tau(x), \tau(g^{-1}x)\}$.

\textbf{Case 1}. $\sharp(\tau\{x, g^{-1}x\}\cap \{x, g^{-1}x\})=0$.

Now, $y:=\tau(x)\not\in \{x, g^{-1}x\}$. Take any 
infinite sequence $\{y_i\}\subseteq X\setminus\{x, g^{-1}x, y\}$, there exists some $t_n\in G$ such that $t_n(x, g^{-1}x, y)=(x, g^{-1}x, y_n)$ for each $n$ as the action is 3-transitive and $X$ is infinite. Then one can check as before that $t_n\tau t_n^{-1}\neq t_m\tau t_m^{-1}$ if $n\neq m$. 

\textbf{Case 2}. $\sharp(\tau\{x, g^{-1}x\}\cap \{x, g^{-1}x\})=1$.

Without loss of generality, we may assume $\tau(x)=x$ or $\tau(x)=g^{-1}x$; otherwise, $\tau(x)\not\in \{x, g^{-1}x\}$ and we can apply the construction in Case 1.

If $\tau(x)=x$ or $\tau(x)=g^{-1}x$, then $z:=\tau(g^{-1}x)\not\in\{ x, g^{-1}x\}$. Now, take any infinite sequence $\{y_n\}\subset X\setminus\{x, g^{-1}x\}$ and find $t_n\in G$ such that  $t_n(x, g^{-1}x, z)=(x, g^{-1}x, y_n)$ for each $n$. Clearly, $t_n\tau t_n^{-1}\neq t_m\tau t_m^{-1}$ if $n\neq m$ as 
$t_m^{-1}t_n\tau (t_m^{-1}t_n)^{-1}(g^{-1}x)\neq \tau (g^{-1}x)$. Obviously, $t_n\in g^{-1}Hg\cap H$.

\textbf{Case 3}. $\sharp(\tau\{x, g^{-1}x\}\cap \{x, g^{-1}x\})=2$.

As $\tau\not\in g^{-1}Hg\cap H$, we know that $\tau(x)=g^{-1}x$ and $\tau(g^{-1}x)=x$.

We split the proof into two subcases according to whether $G$ contains the involution $\sigma$ as in the claim.

\textbf{Subcase 1}. $\sigma\not\in G$.

This implies there exists some $y\in X\setminus \{x, g^{-1}x\}$ such that $\tau(y)\neq y$. Take any infinite sequence $\{y_n\}\subset X\setminus \{x, g^{-1}x, y\}$. We can find $t_n\in G$ such that $t_n(x, g^{-1}x, y, \tau(y))=(x, g^{-1}x, y, y_n)$ for each $n$. Then $t_n\tau t_n^{-1}\neq t_m\tau t_m^{-1}$ if $n\neq m$.

\textbf{Subcase 2}. $\sigma\in G$.

By assumption, $id\neq \tau\neq \sigma$. Define $\tau':=\sigma\tau\sigma^{-1}$. Note that $\tau\sigma^{-1}\in g^{-1}Hg\cap H$ and $(\tau\sigma^{-1})^{-1}\tau(\tau\sigma^{-1})=\tau'$.  
Moreover, $\tau'\neq \sigma$ while $\tau'(x)=\sigma(x)=g^{-1}x$ and $\tau'(g^{-1}x)=\sigma(g^{-1}x)=x$, this implies there exists some $y\in X\setminus \{x, g^{-1}x\}$ such that $\tau'(y)\neq y=\sigma(y)$ as the action is faithful. Now, we can replace $\tau$ by $\tau'$ and run the same argument as in Subcase 1.
\end{proof}
Now, we use the above claim to finish the proof as follows.

If $\sigma\not\in G$, then $u_g^*E(u_g)\in \mathbb{C}1$ and we can follow the proof of Lemma \ref{lem: 2nd approach} to deduce $P=LH$ or $LG$.

If $\sigma\in G$, then the claim tells us that $u_g^{-1}E(u_g)\in L(\{e, \sigma\})=\mathbb{C}1+\mathbb{C}u_{\sigma}$. Now, observe that $G=H\sqcup HgH=H\sqcup H\sigma H$ and $g^{-1}x=\sigma^{-1}x$, so we can actually replace $g$ by $\sigma$ in the proof of the Claim. Therefore, we also get $u_{\sigma}^{-1}E(u_{\sigma})\in \mathbb{C}1+\mathbb{C}u_{\sigma}$.

Next, write $E(u_{\sigma})=u_{\sigma}(a+bu_{\sigma})$ for some $a, b\in \mathbb{C}$, i.e. $E(u_{\sigma})=au_{\sigma}+b1$.
Take trace on both sides, we get $b=0$, hence $au_{\sigma}=E(u_{\sigma})=E(E(u_{\sigma}))=a^2u_{\sigma}$, i.e. $a=0$ or $1$. So $P=LH$ or $LG$.
\end{proof}


Using Theorem \ref{thm: faithful 4-transitive actions}, we can prove Corollary \ref{cor: acylindricially hyperbolic and t.full groups}.

\begin{proof}[Proof of Corollary \ref{cor: acylindricially hyperbolic and t.full groups}]
It suffices to check any $G$ in the corollary admits a faithful, highly transitive action. For the first class of groups, this is due to Hull-Osin \cite[Theorem 1.2]{ho}. For the second class of groups, we can simply take the infinite set to be any (infinite) orbit (in the Cantor set) and consider the action of the topological full group on this orbit.
\end{proof}

We record several remarks on Theorem \ref{thm: faithful 4-transitive actions} and Corollary \ref{cor: acylindricially hyperbolic and t.full groups}.

\begin{remark}
(1) Corollary \ref{cor: acylindricially hyperbolic and t.full groups} covers many groups. Indeed, for the first class of groups,  Hull-Osin's Theorem 1.2 in \cite{ho} generalizes a long list of previous results (see \cite{ho} for explicit discussion on this), notably about finitely generated free groups \cite{mcd, dix, ols_highly transitive}, surface groups \cite{kit}, hyperbolic groups \cite{cha}, free products \cite{fms, gun, hic, ms} or outer automorphism groups for free groups \cite{gg_highly transitive}. Some groups in the second class can be finitely generated and simple \cite{matui, nek_2}, including some which are amenable \cite{jm} or have intermediate growth \cite{nek_1}.

(2) It is not hard to check for a faithful, 3-transitive action on an infinite set, the stabilizer subgroup of any point is relative I.C.C. in the ambient group, but it is not clear whether Theorem \ref{thm: faithful 4-transitive actions} still holds if we weaken the assumption to having a faithful, 3-transitive action (e.g. those in \cite{lm}). It does hold for the following concrete example. Let $G=PSL_2(\mathbb{Q})$, consider the well-known faithful, 3-transitive action $G\curvearrowright X:=\mathbb{P}(\mathbb{Q}^2)$ (the projective line) (\cite[Example 1.16]{ggs}) and set $x=[1, 0]\in X$. Note that the Claim in the above proof fails for this example. Nevertheless, one can still argue that $L(Stab(x))$ is maximal in $L(PSL_2(\mathbb{Q}))$. See Proposition \ref{prop: maximal for PSL_2(Q)} below for details. Moreover, one can check easily that for the faithful, 3-transitive action as mentioned in \cite[P. 347]{ho} (i.e. the affine action $V\rtimes A\curvearrowright V$, where $V:=\oplus_{\mathbb{Z}}\mathbb{Z}/2\mathbb{Z}$ and $A$ is the group of automorphims of $V$ with finite supports), the Claim in the above proof still holds, so for this 3-transitive action, the stabilizer subgroups still generate maximal von Neumann subalgebras.

(3) Things get even ``worse" if we only assume the action is faithful, 2-transitive. Indeed, for a sharply 2-transitive action (see \cite{gg_sharply 2 transitive, rst, rt} for examples), $g^{-1}Hg\cap H$ is always trivial for any $g\not\in H$. So the above proof does not work. In fact, $LH$ is not maximal for the faithful, sharply 2-transitive (affine) action $(\mathbb{Q}, +)\rtimes \mathbb{Q}^{\times}\curvearrowright (\mathbb{Q}, +)$, where $H:=Stab(0)=\mathbb{Q}^{\times}$. 

For a general faithful, 2-transitive but not necessarily sharply 2-transitive action, stabilizer subgroups of points can still fail to generate maximal von Neumann subalgebras. For example, let $G$ be an infinite (necessarily simple) group with exactly two conjugacy classes \cite[Exercise 11.78]{gtm148}, then the left-right shift action $G\times G\curvearrowright G$ is 2-transitive (but never sharply 2-transitive), and as shown before, for the diagonal subgroup $\Delta(G)=Stab(e)$, $L(\Delta(G))$ is not maximal inside $L(G\times G)$.  
  
For certain faithful, 2-transitive (e.g. some in \cite[Theorem 1.10]{ggs}) but not sharply 2-transitive actions, it may still be possible to prove maximality for von Neumann subalgebras generated by the stabilizer subgroups, or at least completely determine all von Neumann subalgebras containing the group von Neumann subalgebra of any stabilizer subgroup. 

(4) There are groups admitting no (or not clear whether admitting) faithful, 4-transitive actions, see \cite{lm} and reference therein. We believe many known maximal subgroups with infinite index for these groups  (e.g. \cite{fg, gs_F, sav_1, sav_2}) still generate maximal von Neumann subalgebras.

It is still open whether $SL_n(\mathbb{Z})$ for $n\geq 3$ admits a faithful, highly transitive action (\cite{gg}, \cite[Question 7.8]{gm}). We do not know whether any maximal subgroups with infinite index in these groups (e.g. those in \cite{ms_1, ms_2, ms_3, gg, gm}) still generate maximal von Neumann subalgebras. 

(5) From the first or second proof of Theorem \ref{thm: solve ge's Q} and Corollary \ref{cor: acylindricially hyperbolic and t.full groups},  we deduce that there exist two maximal subgroups with infinite index of $F_2$, say $H_1$ and $H_2$, satisfying the following conditions.
\begin{itemize}
\item Both $LH_1$ and $LH_2$ are maximal in $LF_2$.
\item $H_1$ contains an infinite normal subgroup of $F_2$, but $H_2$ contains no nontrivial normal subgroups of $F_2$.
\end{itemize}
It is natural to ask whether there exists a $*$-automorphism $\phi$ on $LF_2$ such that $\phi(LH_1)=LH_2$. As explained below, the general answer is no. 

To see this, it suffices to argue that we can further assume $|H_1\backslash F_2/H_1|=\infty$ and $|H_2\backslash F_2/H_2|<\infty$. Once we know that, it would imply that $\ell^2F_2$ has distinct bimodule structures w.r.t. $LH_1$ and $LH_2$ and hence no such $\phi$ exists.

First, any $H_2$ obtained in Corollary \ref{cor: acylindricially hyperbolic and t.full groups} satisfies $|H_2\backslash F_2/H_2|=2$. Then, recall that we can take $H_1=\pi^{-1}(K)$, where $\pi: F_2\twoheadrightarrow \Gamma$ denotes a surjective but non-injective homomorphism and $K$ is any maximal subgroup with infinite index in $\Gamma$. Since $|K\backslash \Gamma/K|\leq |H_1\backslash F_2/H_1|$, it suffices to argue we can assume $|K\backslash \Gamma/K|=\infty$.  For this, one can  take $\Gamma$ to be an infinite Tarski monster group, which is a 2-generated simple group with all proper subgroups being finite \cite{ol_1, ol_2}.

Therefore, we have shown there exist at least two different maximal subfactors $LF_{\infty}$ inside $LF_2$.
\end{remark}

\section{Concluding remarks}\label{section: remarks}

In this section, we record some further remarks related to this work.

\subsection{Maximality v.s. rigidity for subgroup von Neumann algebras}

One may have noticed that to disprove that a von Neumann subalgebra $N$ is maximal in $M$, it suffices to find some $\phi\in Aut(M)$ such that $N\lneq Fix(\phi)\lneq M$. This strategy is applied above to show $L(\Delta(G))$ is not maximal inside $L(G\times G)$ for any nontrivial simple group $G$. Another example is shown in Proposition \ref{prop: upper trangular matices is not maximal}.

It is natural to ask whether the above strategy always works for group von Neumann algebras $LH<LG$, e.g. is it true that $LH$ is maximal in $LG$ (under the inclusion) if $LH$ is a rigid subalgebra in $LG$? 

Here, for von Neumann algebras $N<M$, $N$ is \textbf{rigid} in $M$ (in the sense of Longo \cite{longo}) if for any $\phi\in Aut(M)$, $\phi|_{N}=id$ implies $\phi=id$. See \cite{longo} for more on this notion. Note that this defintion differs from the one in \cite{hamana, yuhei_19}, where (normal) c.p. maps are considered instead of automorphisms.

For the above question, one can easily construct counterexamples using the following proposition.

\begin{proposition}\label{prop: observation to show rigidity does not imply maximality}
Let $G$ be an I.C.C. group with $H<G$ being a normal and relative I.C.C. subgroup. Then $LH$ is rigid in $LG$ if the abelianization $(G/H)_{ab}$ is trivial.
\end{proposition}
\begin{proof}
Take any $\phi\in Aut(LG)$ and $\phi|_{LH}=id$. For any $g\in G$ and $h\in H$, $\phi(u_{ghg^{-1}})=u_{ghg^{-1}}$. This implies that $u_g^*\phi(u_g)\in LH'\cap LG=\mathbb{C}1$, i.e. $\phi(u_g)=u_g\lambda_g$ for some group homomorphism $\lambda: G\to \mathbb{T}$. As $\lambda|_H=1$ and $\mathbb{T}$ is abelian, we deduce $\lambda$ factors through $(G/H)_{ab}$, hence $\lambda\equiv 1$ and $\phi=id$.
\end{proof}
Now, one can take $H=ker(F_n\twoheadrightarrow SL_3(\mathbb{Z}))$ and $G=F_n$ for some finite $n$. By Proposition \ref{prop: observation to show rigidity does not imply maximality}, it is clear that $LH$ is rigid. It is obviously not maximal in $LG$.
However, we do not know any examples of $H\leq G$ such that $LH$ is maximal but not rigid in $LG$. Below, we prove that the maximal von Neumann subalgebras which appeared in Theorem \ref{thm: faithful 4-transitive actions} and Theorem \ref{thm: general version of i.c.c. hyperbolic groups} are also rigid.



\begin{proposition}\label{prop: maximal subgroups in 4-transitve actions case give rise to rigid subalgebras}
Let $G$ and $H$ be the ambient groups and their maximal subgroups respectively in Theorem \ref{thm: faithful 4-transitive actions}. Then $LH$ is rigid in $LG$.
\end{proposition}
\begin{proof}
Let $\phi\in Aut(LG)$ and $\phi|_{LH}=id$. We want to show $\phi=id$. 

Fix any $g\in G\setminus H$. For any $h\in g^{-1}Hg\cap H$, since $ghg^{-1}\in H$, we get $\phi(u_g)u_h\phi(u_g)^{-1}=\phi(u_gu_hu_{g^{-1}})=u_gu_hu_{g^{-1}}$, i.e.
$u_g^{-1}\phi(u_g)\in L(g^{-1}Hg\cap H)'\cap LG$.
Recall that $\sigma$ denotes the involution on $X$ which swaps $x$ and $g^{-1}x$.

\textbf{Case 1}. The involution $\sigma\not\in G$.

By the proof of Theorem \ref{thm: faithful 4-transitive actions}, we know $\phi(u_g)=\lambda_gu_g$ for some $\lambda_g\in \mathbb{C}$. Clearly, $\lambda_g\in \mathbb{T}$. Then observe that $G=H\sqcup HgH$ implies that there exists some $h\in H$ such that $ghg\in HgH$. Assume not, then $gHg\subseteq H$ and therefore $G=gGg=gHg\sqcup gHgHg\subseteq H\cup HHg=H\sqcup Hg$, contradicting $[G:H]=\infty$. Now, write $ghg=h_1gh_2$. Apply $\phi$ to both sides, we deduce that $\lambda_g^2=\lambda_g\in\mathbb{T}$, hence $\lambda_g=1$.
Clearly, this implies that $\phi=id$.

\textbf{Case 2}. $\sigma\in G$.

In this case, we have $u_{\sigma}^{-1}\phi(u_{\sigma})\in \mathbb{C}1+\mathbb{C}u_{\sigma}$.
Write $\phi(u_{\sigma})=u_{\sigma}(a+bu_{\sigma})$. From $1=\phi(u_{\sigma}^2)=\phi(u_{\sigma})\phi(u_{\sigma})$, it is not hard to deduce that $\phi(u_{\sigma})=\pm u_{\sigma}$. Use again the observation that there exists some $h\in H$ such that $\sigma h\sigma\in H\sigma H$ to deduce that we must have $\phi(u_{\sigma})=u_{\sigma}$. Hence, $\phi=id$.
\end{proof}

\begin{proposition}\label{prop: maximal subgroups in thm 2 give rise to rigid subalgebras}
Let $G$ and $H$ be the ambient groups and their maximal subgroups respectively in Theorem \ref{thm: general version of i.c.c. hyperbolic groups}. Then $LH$ is rigid in $LG$.
\end{proposition}
\begin{proof}
From the proof of Theorem \ref{thm: general version of i.c.c. hyperbolic groups}, we know $H$ contains an infinite normal subgroup of $G$, denoted by $K$, such that $K<G$ is relative I.C.C.. Let $\phi\in Aut(LG)$ satisfy $\phi|_{LH}=id$.

For all $k\in K$ and $g\in G$, since $gkg^{-1}\in H$, we deduce that $\phi(u_g)u_k\phi(u_g)^*=u_gu_ku_g^*$. Hence, 
$u_g^*\phi(u_g)\in LK'\cap LG=\mathbb{C}1$. Now, we can write $\phi(u_g)=u_g\lambda(g)$ for some group homomorphism 
$\lambda: G\to \mathbb{T}$ satisfying that $\lambda|_H=1$. Our goal is to show $\lambda\equiv 1$.

We claim that the commutator subgroup $[G, G]\not\leq H$.
Once we know that, we can take any $g\in [G, G]\setminus H$. Since $G=\langle H, g \rangle$ and $\lambda(g)=1$, we deduce that $\lambda\equiv 1$. 

To prove the claim, assume $[G, G]\leq  H$ instead. Then as $[G: H]=\infty$ and $H$ is maximal, we deduce that $H/[G, G]$ is a maximal subgroup inside the abelian group $G/[G, G]$ with infinite index, which is absurd since it is easy to check every maximal subgroup in an abelian group has finite index (using the fact that every subgroup in an abelian group is normal).
\end{proof}

\subsection{Two more examples}
In this subsection, we present two more examples that do not seem to fit into our theorems. Both of them can be handled by ad. hoc. approaches.
\begin{proposition}\label{prop: sl_2z is maximal in sl_2(z[1/p])}
Let $H=SL_2(\mathbb{Z})$ and $G=SL_2(\mathbb{Z}[\frac{1}{p}])$, where $p$ is a prime number. Then $L\bar{H}$ is both maximal and rigid in $L\bar{G}$, where $\bar{H}=PSL_2(\mathbb{Z})$ and $\bar{G}=PSL_2(\mathbb{Z}[\frac{1}{p}])$.
\end{proposition}
\begin{proof}
We first prove $L\bar{H}$ is maximal in $L\bar{G}$.

As mentioned in \cite[Remark 3.3]{cornulier}, the pair $(G, H)$ is a Hecke pair, i.e. for all $g\in G$, $[H: H\cap gHg^{-1}]<\infty$. Moreover, it is well-known (see e.g. \cite[Section 9]{gg}) that $H$ is a maximal subgroup inside $G$ with infinite index. 

Indeed, it is routine to do a matrix calculation to show that $\forall g\in G\setminus H$, we have $\langle H, g\rangle=\langle H,s \rangle=G$, where $s=\begin{pmatrix}
p&0\\
0&1/p
\end{pmatrix}$.

From above, we can deduce that the pair $(\bar{G}, \bar{H})$ is also a Hecke pair and $\bar{H}$ is also maximal inside $\bar{G}$.

Next, one can check that $\bar{H}$ is relative I.C.C. in $\bar{G}$, which implies that $\bar{H}\cap \bar{g}\bar{H}\bar{g}^{-1}$ is also relative I.C.C. in $\bar{G}$ for all $g\in G$ as $(\bar{G}, \bar{H})$ is a Hecke pair.

Now, let $P$ be any intermediate von Neumann subalgbera between $L\bar{H}$ and $L\bar{G}$ and $E: L\bar{G}\twoheadrightarrow P$ be the conditional expectation. For any $g\in G$, we have $u_{\bar{g}}^*E(u_{\bar{g}})\in L(\bar{H}\cap \bar{g}^{-1}\bar{H}\bar{g})'\cap L\bar{G}=\mathbb{C}1$. So the same proof as in Lemma \ref{lem: 2nd approach} shows $L\bar{H}$ is maximal in $L\bar{G}$.

Next, we show $L\bar{H}$ is also rigid in $L\bar{G}$.

Let $\phi\in Aut(L\bar{G})$ satisfy $\phi|_{L\bar{H}}=id$. We will show $\phi=id$.

For any $\bar{h}\in \bar{H}\cap \bar{s}\bar{H}\bar{s}^{-1}$, we have $\bar{s}^{-1}\bar{h}\bar{s}\in \bar{H}$. This implies that $\phi(u_{\bar{s}})u_{\bar{s}}^*\in L(\bar{H}\cap \bar{s}\bar{H}\bar{s}^{-1})'\cap L\bar{G}=\mathbb{C}1$. Hence, $\phi(u_{\bar{s}})=\lambda_{\bar{s}}u_{\bar{s}}$ for some $\lambda_{\bar{s}}\in \mathbb{T}$. Next, observe that $sHs\cap HsH\neq \emptyset$. Indeed, one can check directly that the following identity holds:
\begin{align*}
s\begin{pmatrix}
1&1\\
p-1&p
\end{pmatrix}s=\begin{pmatrix}
1&p\\
-1&1-p
\end{pmatrix}s\begin{pmatrix}
1-p^2& -1\\
p^3+p^2-p& p+1
\end{pmatrix}.
\end{align*}
Therefore, $\lambda_{\bar{s}}^2=\lambda_{\bar{s}}\in \mathbb{T}$, i.e. $\lambda_{\bar{s}}=1$ and $\phi(u_{\bar{s}})=u_{\bar{s}}$. As $G=\langle H, s \rangle$, we deduce that $\phi=id$.
\end{proof}

Next, we show one more example of a maximal subgroup which generates a non-maximal von Neumann subalgebra.

Let $G$ be $SL_2(\mathbb{Q})$ and $H$ be the subgroup of upper triangular matrices. Recall that $G=H\sqcup HsH$, where $s=\begin{pmatrix}
0&1\\
-1&0
\end{pmatrix}$. Hence $H$ is maximal in $G$. 

\begin{proposition}\label{prop: upper trangular matices is not maximal}
$LH$ is neither maximal nor rigid inside $LG$. 
\end{proposition}

\begin{proof}
Let $K=s^{-1}Hs\cap H$. A calculation shows the following hold:
\begin{align}\label{eq: calculate rel. I.C.C. for SL 2(Q)}
\begin{split}
&K=\{\begin{pmatrix}
r&0\\
0&1/r
\end{pmatrix}: 0\neq r\in\mathbb{Q}  \}.\\
&\mbox{For any} ~g=\begin{pmatrix}
a&b\\
c&d
\end{pmatrix}\in G , ~\#\{kgk^{-1}: k\in K\}<\infty~\mbox{iff}~ b=c=0. 
\end{split}
\end{align}
Now, for any intermediate subalgebra $P$ between $LH$ and $LG$, denote by $E$ the conditional expectation from $LG$ onto it. Then write $x_g=u_g^*E(u_g)$ for all $g\in G$. Clearly, $x_{hgh'}=x_{gh'}=u_{h'}^{-1}x_gu_{h'}$ for all $h, h'\in H$.

We aim to construct some $P$ such that $LH\lneq P\lneq LG$. To do this, we first show every $P$ can be described in a ``concise'' form.

Note that for all $h\in K$, $shs^{-1}\in H$. Hence $u_hx_su_h^*=x_s$. Therefore, $x_s\in (LK)'\cap LG\subseteq LK$ by (\ref{eq: calculate rel. I.C.C. for SL 2(Q)}) above.

Now, one has $x_s\in LK$, i.e. $E(u_s)=u_sa$ for some $a\in LK$. Apply $E(\cdot)$ to both sides, we get $u_sa=E(u_s)=E(E(u_s))=E(u_s)a=u_sa^2$, so $a=a^2$. Moreover, notice that $LK$ is abelian, hence by functional calculus, $LK\cong L^{\infty}(X)$ and $a=a^2$ implies $a(x)\in \{0 ,1\}$ for a.e. $x\in X$, i.e. $a$ is a projection. In particular, $a=a^*$.

Now, from $s^{-1}=(-id)s$, where $id$ denotes the $2\times 2$ identity matrix, we deduce that 
\[u_{s^{-1}}a=u_{-id}u_sa=u_{-id}E(u_s)=E(u_{(-id)s})=E(u_{s^{-1}})=(u_sa)^*=au_{s^{-1}}.\]
 Hence, $au_s=u_sa$.

To sum up, we have proved $a\in LK$ is a projection which commutes with $u_s$. It is also clear that $P=\{LH, u_sa\}''$.

This shows that every intermediate von Neumann subalgebra can be written as $P=\{LH, u_sa\}''$ for some projection $a\in LK$ which commutes with $u_s$. But it is not clear in general whether $P=LG$ or not for a given nonzero choice of $a$.  However, if we take $a=\frac{u_{id}+u_{-id}}{2}\in \mathcal{Z}(LG)$ and define $P=\{LH, u_sa\}''$, then we claim $LH\lneq P\lneq LG$.

Clearly, $LH\neq P$. We are left to check that $P\neq LG$.

To see this, first notice that (using the fact that $a$ lies in the center of $LG$) $L^2(P)$ can be written as follows:
\begin{align*}
\bigg\{\bigg(\sum_{h\in H}a_hu_h+\sum_{(h_1, h_2)\in H^2}c_{(h_1, h_2)}u_{h_1}u_sau_{h_2}\bigg)\delta_{id}~\bigl\vert~ a_h, c_{(h_1, h_2)}\in \mathbb{C}\bigg\}\cap \ell^2G.
\end{align*}

Using this description, it is easy to check $L^2(P)$ is a proper subspace of $\ell^2(G)$ as it is orthogonal to the nonzero vector $\delta_s-\delta_{-s}$, so $LH\lneq P\lneq LG$.

Another way to argue $P\neq LG$ is to consider the map $\phi: G\to LG$ defined by $\phi(u_h)=u_h, \forall h\in H$  and $\phi(u_s)=u_{s^{-1}}(1-a)+u_sa$. Now, one argues that this $\phi$ extends to an automorphism of $LG$. For this, one only needs to check $\phi$ preserves all possible identities $u_su_{h_1}u_s=u_{h_2}u_su_{h_3}$ once $sh_1s=h_2sh_3$ holds for some $h_i\in H$.
This is clear since $a$ lies in the center of $LG$. And it is easy to verify that $\phi(u_s)\neq u_s$, so $LG\neq P$ as $\phi|_P=id$. This also shows $LH$ is not rigid inside $LG$.
\end{proof}

In the above example, we can mod out the center $\{\pm id\}$ and consider the maximal subgroup $\bar{H}:=H/\{\pm id\}$ of $\bar{G}:=PSL_2(\mathbb{Q})$, where $H$ denotes the subgroup of upper triangular matrices of $G=SL_2(\mathbb{Q})$. Following the above proof, if we define $a$ in the same way, then we just get the identity in $L(\bar{G})$. So the above argument no longer works. In fact, it turns out  $L(\bar{H})$ is both maximal and rigid in $L(\bar{G})$.

\begin{proposition}\label{prop: maximal for PSL_2(Q)}
Under the above notations, $L(\bar{H})$ is both maximal and rigid inside $L(\bar{G})$.
\end{proposition}

\begin{proof}
With a little abuse of notation, we will always use the same letter to denote both an element in $G$ and its image in $\bar{G}$. 

Let us first prove $L(\bar{H})$ is rigid in $L(\bar{G})$.

Let $\phi\in Aut(L\bar{G})$ satisfy $\phi|_{L(\bar{H})}=id$. We want to show that $\phi=id$.

For any $h\in s^{-1}\bar{H}s\cap \bar{H}=\bar{K}$, we have $shs^{-1}\in \bar{H}$, hence $\phi(u_s)u_h\phi(u_s)^{-1}=u_su_hu_s^{-1}$, and we deduce that $u_s^{-1}\phi(u_s)\in L(s^{-1}\bar{H}s\cap \bar{H})'\cap L\bar{G}=L(\bar{K})'\cap L(\bar{G})=L(\bar{K})$. 

Now, write $\phi(u_s)=u_sa$ for some $a\in \mathcal{U}(L(\bar{K}))$.

Since $s^2=id\in \bar{G}$, we get $id=\phi(u_s)^2=(u_sa)^2$ and $u_s^2=id$. So $u_s=au_sa$.

For any $h$, $h_1$, $h_2$ in $\bar{H}$ satisfying $shs=h_1sh_2$ in $\bar{G}$, we get $\phi(u_su_hu_s)=\phi(u_{h_1}u_su_{h_2})$. We compute both sides to get the following:
\begin{align*}
\phi(u_su_hu_s)&=u_sa(u_hu_s)a\\
&=(u_sa)u_s^{-1}u_{h_1sh_2}a~  \mbox{(Use $u_{hs}=u_{s^{-1}h_1sh_2}$)}\\
&=a^{-1}u_su_s^{-1}u_{h_1sh_2}a ~\mbox{(Use $u_sa=a^{-1}u_s$)}\\
&=a^{-1}\sigma_{h_1sh_2}(a)u_{h_1sh_2}~\mbox{(Here, $\sigma_g:=Ad(u_g),\forall g\in \bar{G}$.)}\\
\phi(u_{h_1}u_su_{h_2})&=u_{h_1}u_sau_{h_2}=\sigma_{h_1s}(a)u_{h_1sh_2}.  
\end{align*}
By comparing the above two expressions, we deduce that
\begin{align}\label{eq: equality on a to show rigidity for PSL_2(Q)}
\sigma_{h_1sh_2}(a)=a\sigma_{h_1s}(a).
\end{align}

Then, notice that for $h=\begin{pmatrix}
1&1\\
0&1
\end{pmatrix}$ and $h_1=h_2=\begin{pmatrix}
1&-1\\
0&1
\end{pmatrix}$, the identity $shs=h_1sh_2$ holds in $\bar{G}$. Moreover, a calculation shows $h_1s=\begin{pmatrix}
1&1\\-1&0
\end{pmatrix}$ and $h_1sh_2=\begin{pmatrix}
1&0\\
-1&1
\end{pmatrix}$. This implies that for any $h=\begin{pmatrix}
t&0\\
0&\frac{1}{t}
\end{pmatrix}\in \bar{H}$, we have 
\begin{align}\label{eq: rotation to show rigidity for PSL_2(Q)}
\sigma_{h_1sh_2}(h)=\begin{pmatrix}
t&0\\
-t+\frac{1}{t}&\frac{1}{t}
\end{pmatrix},\quad
\sigma_{h_1s}(h)=\begin{pmatrix}
\frac{1}{t}&-t+\frac{1}{t}\\
0&t
\end{pmatrix}.
\end{align}
Write $a=\sum_{t\in\mathbb{Q}^+}\lambda_tu_{p_t}\in \mathcal{U}(L(\bar{K}))$, where $p_t=\begin{pmatrix}
t&0\\
0&\frac{1}{t}
\end{pmatrix}.$
By (\ref{eq: equality on a to show rigidity for PSL_2(Q)}) and (\ref{eq: rotation to show rigidity for PSL_2(Q)}), we deduce that 
\begin{align}\label{eq: locate rotation of a to show rigidity of PSL_2(Q)}
\sigma_{h_1sh_2}(a)=a\sigma_{h_1s}(a)\in L(\overline{Upper})\cap L(\overline{Lower})=L(\bar{K}).
\end{align}
Here, $Upper$ and $Lower$ denote the subgroups of $G$ consisting of all upper triangular matrices and lower triangular matrices respectively. 

Denote by $E': L(\bar{G})\twoheadrightarrow L(\bar{K})$ the conditional expectation onto $L(\bar{K})$.
Then, we know that
$\sigma_{h_1sh_2}(a)=E'(a\sigma_{h_1s}(a))=aE'(\sigma_{h_1s}(a))=a\lambda_1$ by (\ref{eq: rotation to show rigidity for PSL_2(Q)}) and (\ref{eq: locate rotation of a to show rigidity of PSL_2(Q)}).
Clearly, this implies that $\lambda_t=0$ for all $t\neq 1$ and $\lambda_1=\lambda_1^2$. As $a$ is a unitary, we get $\lambda_1=1$, so  $a=id$ and $\phi(u_s)=u_s$. Since $\bar{G}=\langle \bar{H},s \rangle$, we deduce $\phi=id$.

Now, let us show $L(\bar{H})$ is maximal in $L(\bar{G})$. The proof is inspired by the proof of \cite[Lemma 3.12]{js}.

Following the proof of Proposition \ref{prop: upper trangular matices is not maximal}, we can still get $E(u_s)=u_sa$ for some projection $a\in L(\bar{K})$ which commutes with $u_s$.

We are left to show that if $a\neq 0$, then $a=id$.

For any $h\in\bar{H}$, note that $E(u_{sh})=u_sau_h$. Once again, assume $shs=h_1sh_2$ holds for some $h_1$, $h_2$ in $\bar{H}$. 
Then we plug these elements into both sides of the identity $E(u_{sh})E(u_s)=E(E(u_{sh})u_s)$ and get the following:
\begin{align*}
E(u_{sh})E(u_s)&=(u_sau_h)(u_sa)=au_{shs}a=au_{h_1sh_2}a=a\sigma_{h_1sh_2}(a)u_{h_1sh_2},\\
E(E(u_{sh})u_s)&=E(u_sau_hu_s)=E(au_su_hu_s)
=aE(u_su_hu_s)\\
&=aE(u_{h_1sh_2})=au_{h_1}u_sau_{h_2}=a\sigma_{h_1s}(a)u_{h_1sh_2}.
\end{align*}
Hence, we get 
\begin{align}\label{eq: equality on a to show maximality in PSL_2(Q)}
a\sigma_{h_1sh_2}(a)=a\sigma_{h_1s}(a).
\end{align}
Now, we claim that the above implies $a=id$.

As in the proof of the rigidity part, we still take the same $h_1$ and $h_2$ as there. And write $a=\sum_{t\in\mathbb{Q}^+}\lambda_tu_{p_t}\in \mathcal{P}(L(\bar{K}))$, where $p_t=\begin{pmatrix}
t&0\\
0&\frac{1}{t}
\end{pmatrix}.$

It is easy to check $au_s=u_sa$ is equivalent to $\lambda_t=\lambda_{1/t}$ for all $t>0$.

From (\ref{eq: rotation to show rigidity for PSL_2(Q)}) and (\ref{eq: equality on a to show maximality in PSL_2(Q)}), we deduce that
\begin{align}\label{eq: locate rotation of a to prove maximality for PSL_2(Q)}
a\sigma_{h_1sh_2}(a)=a\sigma_{h_1s}(a)\in L(\overline{Upper})\cap L(\overline{Lower})=L(\bar{K}).
\end{align}
Here, $Upper$ and $Lower$ denote the subgroups of $G$ consisting of all upper triangular matrices and lower triangular matrices respectively. 
Denote by $E': L(\bar{G})\twoheadrightarrow L(\bar{K})$ the conditional expectation onto $L(\bar{K})$.

Fix any $1\neq t_0\in \mathbb{Q}^+$ and set $q=\begin{pmatrix}
1&t_0^2-1\\
0&1
\end{pmatrix}\in \bar{G}$. 

On the one hand, we have
\begin{align*}
\langle a\sigma_{h_1s}(a)u_q\delta_{id}, \delta_{id} \rangle&=\langle \sigma_{h_1s}(a)u_q\delta_{id}, a\delta_{id} \rangle=\langle \lambda_{t_0}u_{p_{1/t_0}}\delta_{id}, a\delta_{id} \rangle\\
&=\lambda_{t_0}\lambda_{1/t_0}=\lambda_{t_0}^2.
\end{align*}
On the other hand, from (\ref{eq: locate rotation of a to prove maximality for PSL_2(Q)}), we also have $a\sigma_{h_1s}(a)=E'(a\sigma_{h_1sh_2}(a))=aE'(\sigma_{h_1sh_2}(a))=a\lambda_1$.
Hence, we also have
\begin{align*}
\lambda_{t_0}^2=\langle a\sigma_{h_1s}(a)u_q\delta_{id}, \delta_{id} \rangle=\langle a\lambda_1u_q\delta_{id}, \delta_{id} \rangle=\lambda_1\langle u_q\delta_{id},a\delta_{id} \rangle=0.
\end{align*}
Therefore, we deduce $\lambda_{t_0}=0$ for all $1\neq t_0\in\mathbb{Q}^+$. As $a$ is a nonzero projection, this implies that $a=id$. 
\end{proof}

\subsection{Two questions}\label{subsection: two endding questions}

We do not know whether the hyperfinite II$_1$ factor $R$ can be embedded into a non-thin II$_1$ factor (e.g. the free group factor $LF_n$ for $n\geq 3$ by \cite[Corollary 4.3]{gepopa}) as a maximal subfactor with infinite Jones index. Motivated by this question and inspired by \cite[Corollary 1]{ozawapopa} and Peterson-Thom conjecture in \cite[p.590, last paragraph]{petersonthom}, we ask the following questions:

\begin{question}\label{question: infinite intersection for max. subgroups in free groups?}
Let $G$ be the non-abelian free group $F_n$ for $2\leq n<\infty$. If $H$ is a maximal subgroup with infinite index in $G$, is $|gHg^{-1}\cap H|=\infty$ for all $g\in G$?
\end{question}
One may also consider the following von Neumann algebra analog. 
\begin{question}\label{question: max subfactor intersects its conjugate diffusely?}
Let $M$ be the free group factor $LF_n$ for $2\leq n<\infty$. If $N$ is a maximal subfactor with infinite Jones index in $M$, is $uNu^*\cap N$ diffuse for all unitary $u$ in $M$? 
\end{question}
Note that Question \ref{question: infinite intersection for max. subgroups in free groups?} has a negative answer if $G=PSL_n(\mathbb{Z})$ for $n\geq 3$ by \cite{gm}.\\
\paragraph{\textbf{Acknowledgements}}
The author was partially supported by the National Science Center (NCN) grant no. 2014/14/E/ST1/00525 and also by the grant 346300 for IMPAN from the Simons Foundation and the matching 2015-2019 Polish MNiSW fund. A support  of the FWO-PAN grant ``von Neumann algebras arising from quantum symmetries" is also acknowledged.
The author thanks Alejandra Garrido, Liming Ge, Adam Skalski, Yuhei Suzuki, Stefaan Vaes for helpful discussions. He is also very grateful to Adam Skalski and Yuhei Suzuki for reading the draft carefully, pointing out inaccuracies and providing many helpful comments. 

\end{document}